\documentclass[10pt, a4paper, reqno]{amsart}
\usepackage{amscd, amsmath, amstext, amsthm, amssymb, amsopn, amssymb, amsxtra, amsfonts}

\usepackage{tikz}
\usetikzlibrary{matrix,arrows,decorations.pathmorphing}
\usepackage{tikz-cd}
\usepackage{fancyhdr}

\usepackage{paralist}
\usepackage[mathcal]{euscript}
\usepackage[english]{babel}
\usepackage[latin1]{inputenc}

\newtheorem{thm}{Theorem}
\newtheorem{lem}[thm]{Lemma}

\newcommand{\Ker}{\operatorname{Ker}\nolimits}
\newcommand{\Cok}{\operatorname{Cok}\nolimits}
\newcommand{\cok}{\operatorname{cok}\nolimits}
\newcommand{\im}{\operatorname{im}\nolimits}
\newcommand{\coim}{\operatorname{coim}\nolimits}
\renewcommand{\Im}{\operatorname{Im}\nolimits}
\newcommand{\Coim}{\operatorname{Coim}\nolimits}

\begin{document}

\title{\vspace{-15pt}Exact couples in semiabelian categories revisited\vspace{-7pt}}
\author{~}

\renewcommand{\thefootnote}{}
\hspace{-1000pt}\footnote{\hspace{5.5pt}2010 \emph{Mathematics Subject Classification}: Primary 18G50; Secondary 18G40.}
\hspace{-1000pt}\footnote{\hspace{5.5pt}\emph{Key words and phrases}: Exact couple, spectral sequence, semiabelian category, quasiabelian category.\vspace{1.6pt}}
\hspace{-1000pt}\footnote{$^{\text{a}}$\,Sobolev Institute of Mathematics, Pr.~Akad.~Koptyuga 4, 630090, Novosibirsk, Russia \& Novosibirsk State University,\linebreak\phantom{x}\hspace{13pt}ul.~Pirogova 2, 630090, Novosibirsk, Russia, eMail: yakop@nsc.math.ru.\vspace{0.75pt}}
\hspace{-1000pt}\footnote{\hspace{-6.8pt}$^{\text{b},1}$\,Corresponding author: Sobolev Institute of Mathematics, Pr.~Akad.~Koptyuga 4, 630090, Novosibirsk, Russia, Phone:\linebreak\phantom{x}\hspace{12.5pt}+7\hspace{1.2pt}(8)\hspace{1.2pt}383\hspace{1.2pt}/\hspace{1.2pt}363\hspace{1.2pt}-\hspace{1.2pt}4648, Fax:\hspace{1.2pt}\hspace{1.2pt}+7\hspace{1.2pt}(8)\hspace{1.2pt}383\hspace{1.2pt}/\hspace{1.2pt}333\hspace{1.2pt}-\hspace{1.2pt}2598, eMail: wegner@math.uni-wuppertal.de.}

\begin{abstract}\vspace{15pt}
Consider an exact couple in a semiabelian category in the sense of Palamodov, i.e., in an additive category in which every morphism has a kernel as well as a cokernel and the induced morphism between coimage and image is always monic and epic. Assume that the morphisms in the couple are strict, i.e., they induce even isomorphisms between their corresponding coimages and images. We show that the classical construction of Eckmann and Hilton in this case produces two derived couples which are connected by a natural bimorphism. The two couples correspond to the a priori distinct co\-ho\-mo\-lo\-gy objects, the left resp.~right cohomology, associated with the initial exact couple. The de\-ri\-va\-tion process can be iterated under additional assumptions. 
\end{abstract}

\maketitle

\begin{picture}(0,0)
\put(114,147){{\sc\MakeLowercase{Yaroslav Kopylov}}\,$^{\text{a}}$ {\sc and} {\sc\MakeLowercase{Sven-Ake Wegner}}\,$^{\text{b}, 1}$}
\put(203,125){June 7, 2014}

\end{picture}

\vspace{-30pt}

\section{Preliminaries}\label{PRE}

The aim of this note is to elaborate rigorously in which way the classical approach of Eckmann, Hilton \cite{EH} for constructing exact couples, see Massey \cite{M}, generalizes to semiabelian categories. For quasiabelian categories this was carried out by Kopylov \cite{K} in 2004; the negative answer to Ra\u{\i}kov's conjecture given around 2008, see Rump \cite{R}, showed however that the class of semiabelian categories is strictly larger. We point out that semiabelian categories appear in different branches of mathematics, see Ko\-py\-lov, Wegner \cite[Section 1]{KW} for more details and references. For recent results on exact couples in non-additive situations we refer to Grandis \cite{G}.
\smallskip
\\In the sequel, $\mathcal{A}$ always denotes a preabelian category, i.e., $\mathcal{A}$ is additive and every morphism $\alpha$ in $\mathcal{A}$ has a kernel and a cokernel. We denote by $\bar{\alpha}\colon \Coim\alpha\rightarrow\Im\alpha$ the canonical morphism. The category $\mathcal{A}$ is semiabelian if the induced morphism $\bar{\alpha}$ is always monic and epic, v.i.z., a bimorphism. We say that a morphism $\alpha$ is strict if $\bar{\alpha}$ is an isomorphism. We say that a kernel $\alpha$ is semistable if all its pushouts along arbitrary morphisms are again kernels. Semi-stable cokernels are defined dually. An exact couple in $\mathcal{A}$ is a diagram of the form
\begin{equation}\label{ex-coup}
\begin{tikzcd}
D \arrow[]{rr}{\alpha}& &D \arrow[]{ld}{\beta}\\
&E \arrow[]{ul}{\gamma}
\end{tikzcd}
\end{equation}
such that $\im\alpha=\ker\beta$, $\im\beta=\ker\gamma$ and $\im\gamma=\ker\alpha$ hold. If $\mathcal{A}$ is semiabelian, it is easy to see that the latter equations are equivalent to $\cok\alpha=\coim\beta$, $\cok\beta=\coim\gamma$ and $\cok\gamma=\coim\alpha$, respectively. E.g., $\coim\beta=\cok\im\alpha=\cok((\im\alpha)\bar{\alpha}\coim\alpha)=\cok\alpha$ and the dual computation yield the first equivalence.

\section{Results}\label{RESULTS}

\begin{thm}\label{THM-1} Let $\mathcal{A}$ be semiabelian and consider the exact couple \eqref{ex-coup}. Assume that $\alpha$, $\beta$ and $\gamma$ are strict. Then the Eckmann-Hilton construction, see Section \ref{EHC}, gives rise to the following two diagrams
\begin{equation}\label{der-coup}
\begin{tikzcd}
D_1 \arrow[]{rr}{\alpha_1}& &D_1 \arrow[]{ld}{\beta_1^-}\\
&E_1^- \arrow[]{ul}{\gamma_1^-}
\end{tikzcd}\hspace{20pt}
\begin{tikzcd}
D_1 \arrow[]{rr}{\alpha_1}& &D_1 \arrow[]{ld}{\beta_1^+}\\
&E_1^+ \arrow[]{ul}{\gamma_1^+}
\end{tikzcd}
\end{equation}
which we call the \emph{left} resp.~the \emph{right derived couple}. The diagrams in \eqref{der-coup} have the following properties.\vspace{3pt}

\begin{compactitem}
\item[(i)] Both diagrams are exact couples.\vspace{3pt}

\item[(ii)] With $\partial=\beta\gamma$ we get $H^-(E,\partial)=E_1^-$ and $H^+(E,\partial)=E_1^+$. Here, the \emph{left} resp.~\emph{right cohomology} is defined via $H^-(E,\partial)=\Cok(\theta\colon\Im\partial\rightarrow\Ker\partial)$ resp.~$H^+(E,\partial)=\Ker(\tau\colon\Cok\partial\rightarrow\Coim\partial)$, where $\theta$ and $\tau$ are the natural maps.\vspace{3pt}

\item[(iii)] There is a canonical bimorphism $\omega\colon E_1^-\rightarrow E_1^+$ satisfying the equations $\omega\beta_1^-=\beta_1^+$, $\gamma_1^-=\gamma_1^+\omega$ and $(\ker\tau)\omega(\cok\theta)=(\cok\partial)(\ker\partial)$. The morphism $\omega$ is uniquely determined by the third equation and it is an isomorphism if $\ker\partial$ or $\cok\partial$ is semistable.
\end{compactitem}
\end{thm}

\begin{thm}\label{THM-2} Let $\mathcal{A}$ be semiabelian and consider the exact couple \eqref{ex-coup}. Assume that $\beta$ and $\gamma$ are strict and that $\ker\gamma$ and $\cok\beta$ are semistable. If the powers $\alpha^k$ are strict for $1\leqslant{}k\leqslant{}n$, then the derivation process of Theorem \ref{THM-1} can be performed $n$ times, i.e., the Eckmann-Hilton construction gives rise to a complete full binary tree of depth $n$ consisting of exact couples.
\end{thm}

\section{Eckmann-Hilton Construction}\label{EHC}

Let $\mathcal{A}$ be preabelian. We fix the exact couple \eqref{ex-coup} and assume that the morphism $\alpha$ is strict, i.e., it has a decomposition $\alpha=\rho\sigma$ with a kernel $\rho$ and a cokernel $\sigma$. Taking the exactness into account it follows that $\rho=\ker\beta$ and $\sigma=\cok\gamma$ and we can consider the diagram
\begin{equation}\label{eck-hil1}
\begin{tikzcd}
D_1  & &D_1 \arrow[tail]{d}{\rho}\\
D \arrow[two heads]{u}{\sigma} \arrow[]{r}{\beta} & E \arrow[]{r}{\gamma}&  D
\end{tikzcd}
\end{equation}
with $D_1=\Im\alpha=\Coim\alpha=\Ker\beta=\Cok\gamma$. From \eqref{eck-hil1} we obtain the following two diagrams
\begin{equation}\label{eck-hil2}
\begin{tikzcd}[column sep=1.1ex, row sep=1.1ex]
D_1\arrow[]{rr}{\beta_1^-}  && E_1^-\arrow[]{rr}{\gamma_1^-} &&D_1 \arrow[equal]{dd}\\
&\text{\footnotesize\rm PO}&&\text{\footnotesize\rm\phantom{PB}}&\\
D\arrow[two heads]{uu}{\sigma}\arrow[]{rr}{\beta'}   && E_{\rho}\arrow[two heads]{uu}[swap]{\sigma'}\arrow[]{rr}{\gamma'}\arrow[tail]{dd}{\rho'} &&D_1 \arrow[tail]{dd}{\rho}\\
&\text{\footnotesize\rm\phantom{PB}}&&\text{\footnotesize\rm PB}&\\
D \arrow[equal]{uu} \arrow[]{rr}{\beta} && E \arrow[]{rr}{\gamma}&&  D
\end{tikzcd}\hspace{30pt}
\begin{tikzcd}[column sep=1.1ex, row sep=1.1ex]
D_1\arrow[]{rr}{\beta_1^+} && E_1^+\arrow[]{rr}{\gamma_1^+}\arrow[tail]{dd}{\rho''}&& D_1 \arrow[tail]{dd}{\rho}\\
&\text{\footnotesize\rm\phantom{PB}}&&\text{\footnotesize\rm PB}&\\
D_1\arrow[equal]{uu}\arrow[]{rr}{\beta''} && E^{\sigma}\arrow[]{rr}{\gamma''} &&D\arrow[equal]{dd} \\
&\text{\footnotesize\rm PO}&&\text{\footnotesize\rm\phantom{PB}}&\\
D \arrow[two heads]{uu}{\sigma} \arrow[]{rr}{\beta} && E \arrow[]{rr}{\gamma}\arrow[two heads]{uu}[swap]{\sigma''}&&  D
\end{tikzcd}
\end{equation}
by the classical construction of Eckmann, Hilton \cite{EH}. We start with the left diagram. We first form the pullback to obtain $\gamma'$ and $\rho'$. Then we apply its universal property and get $\beta'$ with $\gamma'\beta'=0$ and $\rho'\beta'=\beta$. Next, we form the pushout to obtain $\sigma'$ and $\beta_1^-$. By the universal property of the pushout we finally get $\gamma_1^-$ with $\gamma_1^-\beta_1^-=0$ and $\gamma_1^-\sigma'=\gamma'$. By \cite[Lem.~2.1(iii)]{KW}, $\rho'$ is a kernel and $\sigma'$ is a cokernel. The right diagram is constructed dually. The diagrams in \eqref{der-coup} are obtained by setting $\alpha_1=\sigma\rho$.
\smallskip
\\Using Lemma's \ref{first} and \ref{second} below, an inspection of \cite[proof of Thm.~1]{K} provides that there is a canonical bi\-mor\-phism $\omega\colon E_1^-\rightarrow E_1^+$ which is uniquely determined by the equations $\omega\beta_1^-=\beta_1^+$, $\gamma_1^-=\gamma_1^+\omega$ and $\rho''\omega\sigma'=\sigma''\rho'$.

\section{Proofs}\label{PROOFS}

In \cite{K}, Lemma's \ref{first} and \ref{second} were proved for a quasiabelian category; the results of \cite{KW} allow for the generalizations given below. The category $\mathcal{A}$ is left [right] semiabelian, if $\bar{\alpha}$ is monic [epic] for every morphism $\alpha$.

\begin{lem}\label{first} Let $\mathcal{A}$ be left semiabelian and let $\alpha$, $\beta$ and $\rho$ be morphisms. If $\alpha=\cok\beta$ and $\im\beta=\ker(\rho\alpha)$ hold then $\rho$ is a monomorphism.
\end{lem}
\begin{proof} Suppose that $\rho x=0$. Consider the pullback
\begin{equation}
\begin{tikzcd}[column sep=1.1ex, row sep=1.1ex]
\cdot \arrow[]{rr}{y}\arrow[]{dd}[swap]{z}&&\cdot \arrow[]{dd}{x}\\
&\text{\footnotesize\rm PB}&\\
\cdot\arrow[]{rr}[swap]{\cok\beta}&&\cdot
\end{tikzcd}
\end{equation}
and compute $\rho(\cok\beta)z=\rho xy=0$. Since $\im\beta=\ker(\rho\cok\beta)$, it follows that there exists a unique morphism $u$ with $z=(\im\beta)u$. Hence $xy=(\cok\beta)z=0$. By \cite[Prop.~3.2(iv$'$)]{KW}, $y$ is an epimorphism. Therefore, $x=0$ and thus $\rho$ is a monomorphism.
\end{proof}

\begin{lem}\label{second} Let $\mathcal{A}$ be right semiabelian and let $\alpha$, $\beta$ and $\rho$ be morphisms. If $\coim\alpha=\cok(\rho\beta)$ holds and $\rho$ is a kernel then $\im\beta=\ker(\alpha\rho)$ follows.
\end{lem}
\begin{proof} We check that $\im\beta$ is a kernel of $\alpha\rho$.  We have $\ker\alpha=\ker\coim\alpha=\ker\cok(\rho\beta)=\im(\rho\beta)=\rho\im\beta$, where we used \cite[Cor.~2.3(i) and Prop.~3.1]{KW} for the last equality. It follows that $\alpha\rho\im\beta=0$.
\smallskip
\\Suppose now $\alpha\rho{}x=0$. Since $\ker\alpha=\rho\im\beta$ holds there exists a unique $v$ such that $\rho{}x=\rho(\im\beta)v$ holds. Since $\rho$ is monic, $x=(\im\beta)v$ follows and $v$ is uniquely determined by the last equation.
\end{proof}

\begin{lem}\label{1.5} Let $\mathcal{A}$ be semiabelian and let
\begin{equation}
\begin{tikzcd}[column sep=1.1ex, row sep=1.1ex]
\cdot \arrow[]{rr}{x}\arrow[]{dd}[swap]{s}&&\cdot \arrow[]{dd}{t}\\
&\text{\footnotesize\rm PO}&\\
\cdot\arrow[]{rr}[swap]{y}&&\cdot
\end{tikzcd}
\end{equation}
be a pushout diagram. Assume that $x$ is strict and $\im x$ is a semistable kernel. Then $y$ is strict and $\im y$ is a semistable kernel.
\end{lem}
\begin{proof} We construct the diagram
\begin{equation}
\begin{tikzcd}[column sep=1.1ex, row sep=1.1ex]
\cdot\arrow[]{rr}{\coim x}\arrow[]{dd}[swap]{s} &&\cdot\arrow[]{dd}{}\arrow[]{rr}{\im x}&&\cdot\arrow[]{dd}{t}\\
&\text{\footnotesize\rm PO}&&\text{\footnotesize\rm\phantom{PO}}&\\
\cdot\arrow[]{rr}[swap]{y_0}&&\cdot\arrow[]{rr}[swap]{y_1}&&\cdot
\end{tikzcd}
\end{equation}
by decomposing  $x=(\im x)\coim x$, forming the pushout of $\coim x$ along $s$ to get $y_0$ and using its universal property to get $y_1$ with $y=y_1y_0$. By Kelly \cite[Lem.~5.1]{Kelly} the right-hand square is a pushout and thus $y_1$ is a semistable kernel by Sieg, Wegner \cite[Rem.~2.3(a)]{SW}. By \cite[Cor.~2.3(i) and Prop.~3.2]{KW} and since $y_0$ is a cokernel by \cite[Lem.~2.1(iii)]{KW} it follows that $\im y= y_1\im (y_0) = y_1$. Finally, $y=y_1y_0$ is strict by Schneiders \cite[Rem.~1.1.2(c)]{S}.
\end{proof}

Now we prove the results stated in Section \ref{RESULTS}.

\begin{proof}\textit{(of Theorem \ref{THM-1})} (i) We show $\im\beta_1^-=\ker\gamma_1^-$, $\im\alpha_1=\ker\beta_1^-$ and $\im\gamma_1^-=\ker\alpha_1$.
\smallskip
\\Lemma \ref{second}, $\cok\beta=\coim\gamma$ and $\beta=\rho'\beta'$ imply that $\im\beta'=\ker(\gamma\rho')$ holds. Hence, $\im\beta'=\ker(\rho\gamma')=\ker\gamma'$ and therefore $\coim\gamma'=\cok\im\beta'=\cok\beta'$ holds. By \cite[Lemma 2.1(iii)]{KW} we have $\cok\beta'=(\cok\beta_1^-)\sigma'$. In addition $\gamma'=\gamma_1^-\sigma'$ holds and $\sigma'$ is a cokernel. Therefore, $\coim\gamma'=(\coim\gamma_1^-)\sigma'$ holds by \cite[Cor.~2.3(ii) and Prop.~3.2]{KW}. We compute $(\cok\beta_1^-)\sigma'=\cok\beta'=\coim\gamma'= (\coim\gamma_1^-)\sigma'$ which yields $\cok\beta_1^-=\coim\gamma_1^-$ since $\sigma'$ is epic. The latter is equivalent to $\im\beta_1^-=\ker\gamma_1^-$.
\smallskip
\\We have $\beta=\rho'\beta'$, $\beta$ is strict and $\rho'$ is a kernel. Thus, $\rho'\beta'=\im(\rho'\beta')\bar{\beta}\coim(\rho'\beta')=\rho'(\im\beta')\bar{\beta}\coim(\rho'\beta')$, where the last equality follows from \cite[Cor.~2.3(i) and Prop.~3.1]{KW} and $\bar{\beta}$ is an isomorphism. We get $\beta'=(\im\beta')\bar{\beta}\coim(\rho'\beta')$ and thus $\beta'$ is strict by \cite[Rem.~1.1.2(c)]{S}. Now we can use \cite[Prop.~3.1(vii)]{KW} to obtain that the morphism $\widehat\sigma$ defined by the equality $\sigma(\ker\beta')=(\ker\beta_1^-)\widehat\sigma$ is an epimorphism. We have $\ker\beta'=\ker(\rho'\beta')=\ker\beta=\rho$ and hence $\alpha_1=\sigma\rho=(\ker\beta_1^-)\widehat\sigma$. Therefore, $\im\alpha_1=\ker\beta_1^-$ by \cite[Cor.~2.3(i) and Prop.~3.1]{KW}.
\smallskip
\\By dualizing the latter paragraph we obtain $\cok\gamma_1^+=\coim\alpha_1$ and thus $\cok\gamma_1^-=\cok(\gamma_1^+\omega)=\cok\gamma_1^+$ with $\omega$ as in the last paragraph of Section \ref{EHC}. Therefore, $\cok\gamma_1^-=\coim\alpha_1$ which is equivalent to $\im\gamma_1^-=\ker\alpha_1$.
\smallskip
\\The exactness of the second couple can be proved dually.
\medskip
\\(ii) Observe that $\cok\theta=\cok(\theta'\colon E\rightarrow\Ker\partial)$ and $\ker\tau=\ker(\tau'\colon \Cok\partial\rightarrow{}E)$ hold, where $\theta'$ and $\tau'$  are the natural maps which are uniquely determined by $\partial=(\ker\partial)\theta'$ and $\partial=\tau'\cok\partial$, respectively.
\smallskip
\\In the Eckmann-Hilton construction we obtained $\rho=\ker\beta$ and $\sigma=\cok\gamma$. By \cite[Lem.~2.1(iii)]{KW} we get $\rho'=\ker(\beta\gamma)$ and $\sigma'=\cok(\beta'\gamma)$. From the first equation it follows that $\partial=\beta\gamma=\rho'\beta'\gamma=\ker(\beta\gamma)\beta'\gamma=(\ker\partial)\beta'\gamma$ holds and thus $\theta'=\beta'\gamma$, i.e., $H^-(E,\partial)=\Cok(\beta'\gamma)$. From the second equation we obtain $E_1^-=\Cok(\beta'\gamma)$.
\smallskip
\\The equality $E_1^+=H^+(E,\partial)$ can be proved dually.
\medskip
\\(iii) With $\theta'$ and $\tau'$ as in the proof of (ii) it follows from \cite[Section 3]{K9} that there exists $m\colon H^-(E,\partial)\rightarrow H^+(E,\partial)$ such that $(\ker\tau')m(\cok\theta')=(\cok\partial)(\ker\partial)$ holds and that $m$ is unique with this property.
\smallskip
\\Also as in the proof of (ii) we have $\rho=\ker\beta$ and $\sigma=\cok\gamma$ and obtain $\rho'=\ker(\beta\gamma)$, $\sigma''=\cok(\beta\gamma)$, $\rho''=\ker(\beta\gamma'')$ and $\sigma'=\cok(\beta'\gamma)$ by \cite[Lem.~2.1(iii)]{KW}. We compute $\partial=\beta\gamma''\sigma''=\beta\gamma''\cok\partial$ and $\partial=\rho'\beta'\gamma=(\ker\partial)\beta'\gamma$. Therefore, $\tau'=\beta\gamma''$ and $\theta'=\beta'\gamma$ hold and thus we obtain the equation $\rho''m\sigma'=\ker(\beta\gamma'')m\cok(\beta'\gamma)=(\ker\tau')m(\cok\theta')=(\cok\partial)(\ker\partial)=(\cok\beta\gamma)(\ker\beta\gamma)=\sigma''\rho'$.
\smallskip
\\To finish the proof we observe that the last paragraph in Section \ref{EHC} implies that there exists $\omega$ such that $\rho''\omega\sigma'=\sigma''\rho'$. Since $\rho''$ is monic and $\sigma'$ is epic it follows that $\omega=m$ and thus we get all remaining properties of $\omega$ by Section \ref{PROOFS} and \cite[Lem.~7]{K9}.
\end{proof}

\begin{proof}\textit{(of Theorem \ref{THM-2})} We show that the morphisms in the two derived couples of Theorem \ref{THM-1} are strict and that $\ker\gamma_1^{\pm}$ and $\cok\beta_1^{\pm}$ are semistable. The result then follows by iteration.
\smallskip
\\We claim that $\beta_1^-$ and $\gamma_1^-$ are strict and that $\ker\gamma_1^-=\im\beta_1^-$ and $\cok\beta_1^-=\coim\gamma_1^-$ are semistable. Remember that by our assumptions $\im\beta=\ker\gamma$ and $\coim\gamma=\cok\beta$ are semistable. In the third paragraph of the proof of Theorem \ref{THM-1}(i) we deduced already that $\beta'$ is strict. We also established already the equality $\rho'\im\beta'=\im\beta$ in the latter proof. By \cite[Prop.~2.4(d)]{SW} we get that $\im\beta'$ is a semistable kernel since this is true for $\im\beta$. By Lemma \ref{1.5} we get that $\beta_1^-$ is strict and that $\im\beta_1^-$ is semistable. On the other hand, the dual version of Lemma \ref{1.5} yields that $\gamma'$ is strict and $\coim\gamma'$ is semistable, since $\gamma$ is strict and $\coim\gamma$ is semistable. We know so far that $\sigma'$ is a cokernel and that $\gamma_1^-\sigma'=\gamma'$ is strict. Thus, $\gamma_1^-\sigma'=\im(\gamma_1^-\sigma')\overline{\gamma'}\coim(\gamma_1^-\sigma')=\im(\gamma_1^-\sigma')
 \overline{\gamma'}(\coim\gamma_1^-)\sigma'$, where the last equality follows from \cite[Cor.~2.3(ii) and Prop.~3.2]{KW} and $\overline{\gamma'}$ is an isomorphism. We get that $\gamma_1^-=\im(\gamma_1^-\sigma')\overline{\gamma'}\coim\gamma_1^-$ is strict by \cite[Rem.~1.1.2(c)]{S}. We proved already that $\coim\gamma'=(\coim\gamma_1^-)\sigma'$ holds and that $\coim\gamma'$ is semistable. The latter implies that $\coim\gamma_1^-$ is semistable by \cite[Prop.~2.4(c)]{SW} and the claim is established.
\smallskip
\\By our assumptions $\alpha^2=\rho\sigma\rho\sigma=\rho\alpha_1\sigma$ is strict, i.e., $\overline{\alpha^2}$ is an isomorphism. We compute $\rho\alpha_1\sigma=\im(\rho\alpha_1\sigma)\overline{\alpha^2}\coim(\rho\alpha_1\sigma)=\rho\im(\alpha_1\sigma)\overline{\alpha^2}\coim(\rho\alpha_1)\sigma$, where we  use \cite[Cor.~2.3 and Thm.~3.3]{KW} and the fact that $\rho$ is a kernel and $\sigma$ is a cokernel, see Section \ref{EHC}. It follows that $\alpha_1=\im(\alpha_1\sigma)\overline{\alpha^2}\coim(\rho\alpha_1)$ is strict by \cite[Rem.~1.1.2(c)]{S}.
\smallskip
\\The statements for the second couple can be proved dually.
\end{proof}

\footnotesize

{\sc Acknowledgements. }The first-named author was partially supported by the Russian Foundation for Basic Research (Grant~12-01-00873-a) and the State Maintenance Program for the Leading Scientific Schools and Junior Scientists of the Russian Federation (Grant~NSh-921.2012.1). The second-named author's stay at the Sobolov Institute of Mathematics at Novosibirsk was supported by a fellowshop within the Postdoc-Program of the German Academic Exchange Service (DAAD) (Grant~521/915-069-52).


\normalsize


\begin{thebibliography}{10}

\bibitem{EH}
B.~Eckmann and P.~J. Hilton, \emph{Exact couples in an abelian category}, J.
  Algebra \textbf{3} (1966), 38--87.

\bibitem{G}
M.~Grandis, \emph{Homotopy spectral sequences}, J. Homotopy Relat. Struct.
  \textbf{5} (2010), no.~1, 213--252.

\bibitem{Kelly}
G.~M. Kelly, \emph{Monomorphisms, epimorphisms, and pull-backs}, J. Austral.
  Math. Soc. \textbf{9} (1969), 124--142.

\bibitem{K}
Ya. Kopylov, \emph{Exact couples in a {R}a\u{\i}kov semi-abelian category},
  Cah. Topol. G\'eom. Diff\'er. Cat\'eg. \textbf{45} (2004), no.~3, 162--178.

\bibitem{K9}
Ya. Kopylov, \emph{Homology in {P}-semi-abelian categories}, Sci. Ser. A Math. Sci.
  (N.S.) \textbf{17} (2009), 105--114.

\bibitem{KW}
Ya. Kopylov and S.-A. Wegner, \emph{On the notion of a semi-abelian category in
  the sense of {P}alamodov}, Appl. Categor. Struct. \textbf{20} (2012), no.~5,
  531--541.

\bibitem{M}
W.~S. Massey, \emph{Exact couples in algebraic topology. {I}, {II}}, Ann. of
  Math. (2) \textbf{56} (1952), 363--396.

\bibitem{R}
W.~Rump, \emph{Analysis of a problem of {R}aikov with applications to barreled
  and bornological spaces}, J. Pure Appl. Algebra \textbf{215} (2011), no.~1,
  44--52.

\bibitem{S}
J.-P. Schneiders, \emph{Quasi-{A}belian categories and sheaves}, M\'{e}m. Soc.
  Math. Fr., Nouv. S\'{e}r. \textbf{76}, 1999.

\bibitem{SW}
D.~Sieg and S.-A. Wegner, \emph{Maximal exact structures on additive
  categories}, Math. Nachr. \textbf{284} (2011), no.~16, 2093--2100.

\end{thebibliography}
\end{document}